\documentclass{amsart}

\usepackage[utf8]{inputenc}
\usepackage[margin=1in]{geometry}
\usepackage{amsmath, amssymb, amsthm}
\usepackage{graphicx}
\usepackage{dsfont}
\usepackage{xcolor}
\usepackage{dirtytalk}
\usepackage{comment}
\usepackage{hyperref}
\usepackage{mathtools}
\usepackage{enumitem}
\usepackage[normalem]{ulem}
\usepackage{tikz}
\usepackage{comment}



\theoremstyle{plain}

\newtheorem*{conjecture}{Conjecture}
\newtheorem{theorem}{Theorem}[section]
\newtheorem{lemma}[theorem]{Lemma}

\newtheorem{preremark}[theorem]{Remark}
\newenvironment{remark}{\begin{preremark}\rm}{\medskip \end{preremark}}

\theoremstyle{definition}
\newtheorem{definition}[theorem]{Definition}

\numberwithin{equation}{section}

\newcommand{\R}{\mathbb R}

\newcommand{\A}{\mathcal A}
\newcommand{\dd}{\, \mathrm{d}}

\DeclareMathOperator*{\Tail}{ {Tail} }

\newcommand{\dps}{(-\Delta_p)^s}

\title[H\"older Gradient Estimates for  $(s,p)$-Poisson]{H\"older Gradient Estimates for  $(s,p)$-Poisson Equations with measurable and Hölder Data}
\author{Davide Giovagnoli and David Jesus}

\newcommand{\Addresses}{{
  \bigskip
  \footnotesize

    Davide Giovagnoli, \textsc{Dipartimento di Matematica,
Universit\`a di Bologna, 
 Bologna, BO 40126, Italy}\par\nopagebreak
  \textit{E-mail address:} {\tt d.giovagnoli@unibo.it}\\

  \medskip

    David Jesus, \textsc{Dipartimento SMFI, Universit\`a di Parma, Parco Area delle Scienze 53/a, Campus, 43124 Parma, Italy}\par\nopagebreak
  \textit{E-mail address:} {\tt davidjoao.brandligtdejesus@unipr.it}\\

 }

}

\hypersetup{hidelinks,pdftitle={Holder gradient estimates for the fractional p-Poisson equation}}
\begin{document}
\setlength{\emergencystretch}{2em}
\begin{abstract}
We prove interior $C^{1,\alpha}$ estimates for weak solutions of the
fractional $p$-Poisson equation, with $p\geq2$. For right-hand sides
in $L^q$, we obtain the estimate when $p<(1-d/q)/(1-s)$. For
$\gamma$-H\"older continuous right-hand sides, the range is
$p<(1+\gamma)/(1-s)$. Both ranges of parameters are optimal. 
\end{abstract}
\maketitle
	\tableofcontents 
    
\section{Introduction}

We consider weak solutions of the equation
\begin{equation}\label{eq:fracplap}
 \dps u=f\quad\text{in }B_2,
\end{equation}
where $s\in(0,1)$, $p\geq2$. Here $d\geq1$, and $B_2 \subset \R^d$ is the ball centered at $0$  with radius $2$. The operator is given by
\[
 \dps u(x)=\operatorname{PV}\int_{\R^d}
 \frac{|u(x)-u(y)|^{p-2}(u(x)-u(y))}{|x-y|^{d+sp}}\,\dd y.
\]

Our first result concerns source terms in Lebesgue spaces. We write $d/q=0$
when $q=\infty$.

\begin{theorem}\label{t:measurable}
Let  $0<s<1$, and $1<q\leq\infty$ satisfy
\begin{equation}\label{eq:measurable-range}
 2 \leq p<\frac{1-d/q}{1-s}.
\end{equation}

Suppose that $u\in W^{s,p}_{\mathrm{loc}}(B_2)
\cap L^{p-1}_{sp}(\R^d)$ is a weak solution of
$\dps u=f$ in $B_2$, with $f\in L^q(B_2)$. Then
$u\in C^{1,\alpha}(B_1)$ and
\begin{equation}\label{eq:measurable-estimate}
 \|u\|_{C^{1,\alpha}(B_1)}
 \leq C\left(\|u\|_{L^\infty(B_2)}
 +{\Tail}_{p-1,sp}(u;2)
 +\|f\|_{L^q(B_2)}^{1/(p-1)}\right),
\end{equation}
where $C$ and $\alpha>0$ depend only on $d,s,p,q$. More precisely,
if $\alpha_0\in(0,1)$ is an exponent in the homogeneous estimate
of \cite[Theorem 1.1]{GJS}, one may take
\begin{equation}\label{eq:measurable-exponent}
 \eta_q=\frac{sp-d/q}{p-1}-1>0,
 \qquad \alpha=\frac{\alpha_0\eta_q}{1+\alpha_0+\eta_q}.
\end{equation}

\end{theorem}

Our next result concerns H\"older continuous source terms.

\begin{theorem}\label{t:holder}
Let $s\in(0,1)$, $0<\gamma\leq1$, and
\begin{equation}\label{eq:holder-range}
 2\leq p<\frac{1+\gamma}{1-s}.
\end{equation}
Assume that $u\in W^{s,p}_{\mathrm{loc}}(B_2)
\cap L^{p-1}_{sp}(\R^d)$ is a weak solution of
\eqref{eq:fracplap}, with $f\in C^{0,\gamma}(B_2)$.
Then $u\in C^{1,\alpha}(B_1)$ for some $\alpha>0$, and
\begin{equation}\label{eq:main-estimate}
 \|u\|_{C^{1,\alpha}(B_1)}
 \leq C\left(\|u\|_{L^\infty(B_2)}
 +{\Tail}_{p-1,sp}(u;2)
 +\|f\|_{C^{0,\gamma}(B_2)}^{1/(p-1)}\right).
\end{equation}
Here $C$ and $\alpha$ depend only on $d,s,p,\gamma$.
More precisely,
if $\alpha_0\in(0,1)$ is an exponent in the homogeneous estimate
of \cite[Theorem 1.1]{GJS}, one may take
\begin{equation}\label{eq:exponent}
 \eta_\gamma=\frac{sp+\gamma}{p-1}-1>0,
 \qquad \alpha=\frac{\alpha_0\eta_\gamma}{1+\alpha_0+\eta_\gamma}.
\end{equation}
\end{theorem}

    In both Theorems \ref{t:measurable}  and \ref{t:holder} the range of $p$ for which solutions are $C^{1,\alpha}$ is optimal. We illustrate this with an example from \cite{brasco2018higher}.

\begin{remark}\label{rmk:optimality}
For $\beta\in(0,2)$, consider
\[
u_\beta(x)=\min\{|x|^\beta,1\}.
\]
For the exponents used below,
$u_\beta\in W^{s,p}(B_1)\cap L_{sp}^{p-1}(\R^d)$.
Writing $f_\beta=\dps u_\beta$, we have
\begin{equation}\label{eq:example}
\dps u_\beta=f_\beta\quad\text{in }B_{1/2}
\end{equation}
in the weak sense, see Section \ref{s:preli} for precise definitions.

First, let $\max\{1,d/(sp)\}<q\leq\infty$ and take
\[
\beta=1+\eta_q+\varepsilon
=\frac{sp-d/q}{p-1}+\varepsilon,
\]
with $\varepsilon>0$ sufficiently small that $0<\beta<2$.
If $q<\infty$, choose also
$\varepsilon<d/(q(p-1))$. Arguing as in \cite{brasco2018higher}, we get
\[
f_\beta(x)
=c_\beta |x|^{-d/q+\varepsilon(p-1)}+g_\beta(x),
\qquad
c_\beta<0,\quad g_\beta\in L^\infty(B_{1/2}).
\]
Consequently, $f_\beta\in L^q(B_{1/2})$.
When $q=\infty$, the inequality $\beta(p-1)>sp$
instead gives $f_\beta\in L^\infty(B_{1/2})$ directly.

Suppose that all solutions with $L^q$ right-hand side
belong to $C^{1,\alpha}_{\mathrm{loc}}$, for some
$\alpha\in(0,1)$ depending only on $d,s,p,q$.
Since $0<\beta<2$, the function $u_\beta$ belongs to
$C^{1,\alpha}$ near zero only if $\beta\geq1+\alpha$.
Thus
\[
\alpha\leq\eta_q+\varepsilon.
\]
Letting $\varepsilon\downarrow0$, we obtain the necessary
bound
\[
\alpha\leq\eta_q=\frac{sp-d/q}{p-1}-1.
\]
In particular, the existence of such a positive exponent
requires
\[
p<\frac{1-d/q}{1-s}.
\]

For H\"older continuous sources, let $0<\gamma\leq1$ and
first suppose that
\[
1<\beta=1+\eta_\gamma
=\frac{sp+\gamma}{p-1}<2.
\]
A direct calculation gives
\[
f_\beta(x)
=f_\beta(0)+c_\beta|x|^\gamma+g_\beta(x),
\qquad g_\beta\in C^{0,1}(B_{1/2}).
\]
Hence $f_\beta\in C^{0,\gamma}(B_{1/2})$, while
$u_\beta\notin C^{1,\alpha}(B_{1/2})$ whenever
$\alpha>\beta-1$. Therefore, any gradient H\"older
exponent valid for all such solutions must satisfy
\[
\alpha\leq\eta_\gamma
=\frac{sp+\gamma}{p-1}-1.
\]

Finally, suppose that $sp+\gamma\leq p-1$.
For any proposed exponent $\alpha\in(0,1)$, choose
$0<\varepsilon<\alpha$ and use $u_{1+\varepsilon}$.
Since
\[
(1+\varepsilon)(p-1)-sp
\geq\gamma+\varepsilon(p-1),
\]
its source belongs to $C^{0,\gamma}(B_{1/2})$,
whereas $u_{1+\varepsilon}\notin C^{1,\alpha}(B_{1/2})$.
This excludes the borderline as well and shows that
\[
p<\frac{1+\gamma}{1-s}
\]
is necessary for a positive gradient H\"older exponent
depending only on the parameters.
\end{remark}

The remark above shows that an exponent $\alpha\in(0,1)$ valid for all solutions of \eqref{eq:fracplap} with $L^q$ or $C^\gamma$ right-hand sides cannot exceed $\eta_q$ or $\eta_\gamma$, respectively. This suggests the following conjecture for a bounded right-hand side, which by analogy with the local case (see \cite{Araujo_Teixeira_Urbano,Araujo_Teixeira_Urbano2}), we will call the $C^{sp'}$-conjecture. A more general conjecture for $L^q$ (see \cite{lindgren2017regularity}) or $C^\gamma$ right-hand sides could also be included. 

\begin{conjecture}[$C^{sp'}$-conjecture]
\label{sp'-conjecture}
~ Every solution to the equation $\dps u = f \in L^{\infty}$ is locally of class $C^{\frac{sp}{p-1}}$ whenever $p \neq \frac{1}{1-s}$.
\end{conjecture}

Here and below, when \(1<\sigma<2\), the notation \(C^\sigma\) stands for \(C^{1,\sigma-1}\).
This conjecture is known to be true for $p>\frac{1}{1-s}$, see \cite[Theorem 1.1]{biswas2025lipschitz}, and its optimality is ensured by Remark \ref{rmk:optimality}.
In the critical case $p = \frac{1}{1-s}$, in  \cite[Theorem 1.2]{biswas2025lipschitz}  the authors prove solutions satisfy a modulus of continuity of the form $Cr(1+|\log r|^\vartheta)$, for some $\vartheta>0$. The Lipschitz endpoint fails already for $p=2$ and $s=1/2$, since there is a solution $u$ of $(-\Delta)^{1/2}u\in L^{\infty}$ which is not Lipschitz, see \cite[Example 1.6]{brasco2018higher}.

In the range $1<p<\frac{1}{1-s}$, that is $\frac{sp}{p-1}>1$, the conjecture is open.

\medskip
\noindent \textbf{Known literature. }For the homogeneous equation in the degenerate range $2\leq p< 2/(1-s)$, $C^{1,\alpha_0}$ regularity was recently obtained in \cite{GJS}. Regarding inhomogeneous equations, the results depend on the assumptions on the source term. When $f\in L^q$ and $1<q<\infty$, H\"older regularity with precise exponents was established in \cite{brasco2018higher,garain2024higher} where the authors show that $u \in C^\theta$ for any $\theta<1$, in the range $1<p<(1-d/q)/(1-s)$.
For Sobolev regularity see the works \cite{bogelein2024gradient,bogelein2024regularity,bogelein2025gradient}. For boundary regularity see \cite{iannizzotto2020fine,iannizzotto2026boundary}.

When the source term is in $L^\infty\cap\, C^{0,\gamma}$ for some $\gamma\in[0,1]$, then Lipschitz regularity was proved in \cite{biswas2025lipschitz,biswas2025improved} for $1<p<(1+\gamma)/(1-s)$.

In the homogeneous degenerate parabolic setting, \cite{JSU}
established existence, uniqueness, equivalence between weak and
viscosity solutions, and Lipschitz regularity. More recently,
\cite{GJS_Caloric} proved $C^{1,\alpha_0}$ regularity.

\medskip
\noindent \textbf{Description of the argument. }  We now describe the idea behind Theorem \ref{t:measurable}. We want to compare the solution $u$ with its homogeneous replacement in a ball $B_R$. Due to the nonlinearity of the operator, the difference of these solutions does not solve the equation with the difference of the sources. We overcome this difficulty, by obtaining a new quantitative comparison estimate in Lemma \ref{l:comparison-Lq}, which we believe is interesting on its own. Using this comparison estimate, we are able to bound the difference of $u$ and its homogeneous replacement by the factor $R^{(sp-d/q)/(p-1)}$ which is superlinear precisely in the range of Theorem \ref{t:measurable}.
Using the known regularity for solutions with measurable data and the \(C^{1,\alpha_0}\) regularity of \((s,p)\)-harmonic functions, we first obtain a Lipschitz bound and then refine it to the desired \(C^{1,\alpha}\) estimate.

Regarding Theorem \ref{t:holder}, the argument is similar to the one above, except we compare $u$ with the solution of the equation with constant source term $f(x_0)$ in a ball $B_R(x_0)$. Using the $C^{0,\gamma}$ regularity of $f$, the error is of order
$R^{(sp+\gamma)/(p-1)}$, which is superlinear precisely in the range
of Theorem~\ref{t:holder}. 

The proof of the homogeneous H\"older gradient regularity of \cite{GJS} adapts to a constant right-hand side with an estimate depending only on the Lipschitz norm and the smaller tail of order $sp+1$ of the solution. This extra power of the tail allows us to obtain the result in the full range in \eqref{eq:holder-range}.

\section{Preliminaries}\label{s:preli}

We use the same tail convention as in \cite{GJS}, namely
\begin{equation}\label{eq:tail}
 {\Tail}_{m,\lambda}(v;x_0,R)
 =\left(\int_{B_R(x_0)^c}
 \frac{|v(y)|^m}{|y-x_0|^{d+\lambda}}\,\dd y\right)^{1/m},
 \qquad m,\lambda>0.
\end{equation}
We omit the center when it is zero.

For \(m,\lambda>0\), we denote by \(L_\lambda^m(\R^d)\) the space of functions in
$L^{m}_{\mathrm{loc}}(\R^d)$ for which 
\[ 
\int_{\R^d}\frac{|v(y)|^m}{(1+|y|)^{d+\lambda}}\,dy<\infty. 
\]

The condition $u \in L_{sp}^{p-1}(\R^d)$ is a standard requirement to make sense of $(-\Delta_p)^s u$. See \cite[Lemma 2.4]{GJS} for some relations about tail functions.

We recall the scaling property of the tail function.
\begin{lemma}\label{l:tail_scale}
Let $m,\lambda>0$ and $u\in L^m_\lambda(\R^d)$.
For $r,R,M>0$ and $x_0\in\R^d$, define
\[
v(x)=\frac{u(Rx+x_0)}{M}.
\]
Then
\[
{\Tail}_{m,\lambda}(v;r)
=\frac{R^{\lambda/m}}{M}
{\Tail}_{m,\lambda}(u;x_0,Rr).
\]
\end{lemma}

\begin{proof}
The change of variables $y=Rx+x_0$ gives
\[
\begin{aligned}
{\Tail}_{m,\lambda}(v;r)^m
&=\frac1{M^m}\int_{B_r^c}
  \frac{|u(Rx+x_0)|^m}{|x|^{d+\lambda}}\,\dd x\\
&=\frac{R^\lambda}{M^m}
  \int_{\R^d\setminus B_{Rr}(x_0)}
  \frac{|u(y)|^m}{|y-x_0|^{d+\lambda}}\,\dd y\\
&=\frac{R^\lambda}{M^m}
  {\Tail}_{m,\lambda}(u;x_0,Rr)^m
\end{aligned}
\]
which gives the desired equality.
\end{proof}

In order to recall the notion of weak solution, we define the following fractional Sobolev space 
\[
W^{s, p}(\Omega)=\left\{ v\in L^p(\Omega)\; :\;[v]_{W^{s,p}(\Omega)}:=\iint_{\Omega\times \Omega} \frac{|v(x)-v(y)|^p}{|x-y|^{d+sp}}\dd x\dd y<\infty \right\},
\]
and the corresponding spaces
\[
W^{s, p}_{loc}(\Omega)=\{ v\in L^p_{loc}(\Omega)\; :\;  v\in W^{s, p}(\widetilde\Omega)\; \text{for any}\; \widetilde\Omega\Subset \Omega \},
\]
and
\[
W_0^{s, p}(\Omega)=\left\{ v \in W^{s, p}(\R^d) \, : \, v=0 \text{ a.e. in }\R^d \setminus \Omega \right\}. 
\]

\begin{definition}\label{DefWeakSol}
A function $u\in W^{s, p}_{ loc}(\Omega)\cap L^{p-1}_{sp}(\R^d)$ is said to be a weak solution of
\[
(-\Delta_p)^s u=f\quad\text{in }\Omega 
\]
if we have
$$ \frac12\iint_{\R^d\times \R^d} J_p(u(x)-u(y))(\varphi(x)-\varphi(y))|x-y|^{-d-sp}\dd x\dd y=\int_{\Omega}f\varphi\,\dd x,$$
for all $\varphi\in C^\infty_c(\Omega)$, where $J_p(t)=|t|^{p-2}t$.
\end{definition}

We will repeatedly use the inequalities
\begin{align}
 (J_p(a)-J_p(b))(a-b)&\geq2^{2-p}|a-b|^p,
 \label{eq:monotonicity}\\
 \big||a+b|^p-|a|^p\big|
 &\leq C_p\bigl(|a|^{p-1}|b|+|b|^p\bigr),
 \label{eq:energy-difference}\\
 |a+b|^p-|a|^p-pJ_p(a)b&\geq2^{2-p}|b|^p.
 \label{eq:convexity}
\end{align}
They hold for all $a,b\in\R$ and $p\geq2$. The first follows by
integrating $J_p'$ along the interval between $a$ and $b$. The second is the mean value
inequality. Integrating the first along the segment from $a$ to $a+b$
gives the third.

We recall also the classical fractional Sobolev-Morrey inequalities, which can be found in \cite[Chapter 7]{Adams}.

\begin{lemma} \label{l:sob_embed}
    For $1<p<\infty$ and $s \in (0,1)$, let $\Omega\subset\mathbb{R}^d$ be an open and bounded set. 
    If \(sp<d\),  define the fractional critical Sobolev exponent by \begin{equation}\label{eq:fractional-critical-exponent} p_s^*:=\frac{dp}{d-sp}. \end{equation}
    For every $u\in W_0^{s,p}(\Omega)$ the following estimates hold 
    \begin{align} 
    \|u\|_{L^{p_s^*}(\Omega)}^p &\leq C [u]_{W^{s,p}(\mathbb{R}^d)}^p, &&\text{if }sp<d, \label{eq:fractional-Sobolev-subcritical}
    \\ 
    \|u\|_{L^q(\Omega)}^p &\leq C |\Omega|^{\frac{p}{q}+\frac{sp}{d}-1} [u]_{W^{s,p}(\mathbb{R}^d)}^p, &&\text{for every \(1\leq q<\infty\), if }sp=d, \label{eq:fractional-Sobolev-critical}
    \\
    \|u\|_{L^\infty(\Omega)}^p &\leq C |\Omega|^{\frac{sp}{d}-1} [u]_{W^{s,p}(\mathbb{R}^d)}^p, &&\text{if }sp>d. \label{eq:fractional-Morrey} \end{align}
    Here the constants $C$ depend only on $d,\, s, \, p$ and $q$.
\end{lemma}

We next recall a well-known property of H\"older spaces, for which we provide a sketch of the proof.

\begin{lemma}\label{l:affine-approximation}
Let $\alpha\in(0,1)$, $R>0$, and
$u\in C(B_{2R}(x_0))\cap L^\infty(B_{2R}(x_0))$.
Assume that, for every $x\in B_R(x_0)$ and $0<r\leq R$,
there is an affine function $\ell_{x,r}$ such that
\begin{align}\label{eq:C1alpha_est}
    \|u-\ell_{x,r}\|_{L^\infty(B_r(x))}
 \leq M r^{1+\alpha}.
\end{align}
Then $u\in C^{1,\alpha}(B_R(x_0))$ and
\[
 R\|\nabla u\|_{L^\infty(B_R(x_0))}
 +R^{1+\alpha}[\nabla u]_{C^{0,\alpha}(B_R(x_0))}
 \leq C\left(
 \|u\|_{L^\infty(B_{2R}(x_0))}+MR^{1+\alpha}\right),
\]
where $C$ depends only on $d$ and $\alpha$.
Moreover,
\[
 |\nabla u(x)-\nabla \ell_{x,r}|\leq CM r^\alpha.
\]
\end{lemma}

\begin{proof}
By translation and scaling, it suffices to consider $x_0=0$
and $R=1$. Adjusting the constant terms, we may assume
$\ell_{x,r}(x)=u(x)$, up to replacing $M$ by $2M$.
Write $a_{x,r}=\nabla \ell_{x,r}$.

Comparing the estimates in \eqref{eq:C1alpha_est} for $\ell_{x,r}$ and $\ell_{x,r/2}$, we obtain
\begin{align*}
    \sup_{B_{r/2}(x)}|(a_{x,r}-a_{x,r/2})\cdot(y-x)| &\leq \sup_{B_{r}(x)}|u(y)-u(x)-a_{x,r}\cdot(y-x)|+\sup_{B_{r/2}(x)}|u(y)-u(x)-a_{x,r/2}\cdot(y-x)|\\
    &\leq CM(r^{1+\alpha}+(r/2)^{1+\alpha})\leq CMr^{1+\alpha},
\end{align*}
which gives
\[
 |a_{x,r}-a_{x,r/2}|\leq CM r^\alpha.
\]
Iterating this argument along dyadic scales $r_j=2^{-j}r$ we find 
\[
|a_{x,r_j}-a_{x,r_{j+1}}|\leq CM 2^{-j\alpha}r^\alpha \quad \text{for all }j\geq0,
\]
and thus $a_{x,r_j}\to a(x)$ for some $a(x) \in \R^d$ with 
\begin{align*}
|a(x)-a_{x,r}|\leq \sum_{j=0}^\infty |a_{x,r_j}-a_{x,r_{j+1}}|\leq CMr^\alpha \sum_{j=0}^\infty 2^{-j\alpha}\leq CMr^\alpha .
\end{align*}
Choosing $j$ such that $r_{j+1}<|h|\leq r_j$, by  \eqref{eq:C1alpha_est}  we obtain 
\[ 
|u(x+h)-u(x)-a(x)\cdot h| \leq CM|h|^{1+\alpha}, 
\] and hence $u$ is differentiable at $x$ with $\nabla u(x)=a(x)$ and
\[
 |\nabla u(x)-a_{x,r}|\leq CM r^\alpha.
\]

Taking $r=1$ also yields
\[
 \|\nabla u\|_{L^\infty(B_1)}
 \leq C\left(\|u\|_{L^\infty(B_2)}+M\right).
\]

If $x,y\in B_1$ and $0<|x-y|\leq1/4$, set $r=4|x-y|$.
Comparing $\ell_{x,r}$ and $\ell_{y,r}$ on
$B_{r/2}(x)\subset B_r(x)\cap B_r(y)$ gives
\[
 |a_{x,r}-a_{y,r}|\leq CM r^\alpha.
\]
Consequently,
\[
 |\nabla u(x)-\nabla u(y)|\leq CM|x-y|^\alpha.
\]
For larger distances, use the gradient supremum bound.
Rescaling proves the stated estimates.
\end{proof}

We also recall the fast convergence lemma.

\begin{lemma}\label{l:fast_convergence}
    Let $(Y_j)_{j}$ be a sequence of positive real numbers satisfying the recursive inequalities
    \[
        Y_{j+1}\leq Cb^{j} Y_j^{1+\beta},
    \]
    where $C,b>1$ and $\beta>0$ are given numbers. If
    \[
        Y_0\leq C^{-1/\beta}b^{-1/\beta^2},
    \]
    then $Y_j\to 0$ as $j\to\infty$.
\end{lemma}

\section{Comparison estimate with same exterior values}
\label{s:comparison}

The following comparison estimate is the main tool of this paper.
It controls the difference of two solutions with the same exterior
values solely in terms of the difference of their right-hand sides. This  is particularly useful for a nonlinear
operator, since the difference of two solutions does not itself
solve the equation with the difference of the source terms.

Taking $g=0$ or $g=f(x_0)$ gives approximation errors of order
$R^{(sp-d/q)/(p-1)}$ and $R^{(sp+\gamma)/(p-1)}$, respectively for $f\in L^q$ and $f\in C^{0,\gamma}$.
These errors are superlinear precisely in the ranges of our main
theorems, allowing us to deduce gradient H\"older estimates from
the corresponding estimates for homogeneous and constant-source
equations. 

The proof uses a Stampacchia truncation argument, see \cite{stampacchia1965probleme}.
For homogeneous replacements, qualitative interior stability under
uniform bounds on the solution and its tail was established in
\cite[Lemma 6.1]{brasco2018higher}.

\begin{lemma}\label{l:comparison-Lq}
Let $p\geq2$, $0<s<1$, and
$\max\{1,d/(sp)\}<q\leq\infty$. Suppose that
$B_R(x_0)\Subset\Omega$ and that
$u\in W^{s,p}_{\mathrm{loc}}(\Omega)\cap L^{p-1}_{sp}(\R^d)$
is locally bounded and solves $\dps u=f$ weakly in $\Omega$, with $f\in L^q(\Omega)$. For every $g\in L^q(B_R(x_0))$ there is a unique
weak solution $h$ of
\[
 \dps h=g\quad\text{in }B_R(x_0),\quad
 h=u\quad\text{in }B_R(x_0)^c,\quad h-u\in W^{s,p}_0(B_R(x_0)).
\]
Moreover,
\begin{equation}\label{eq:comparison-Lq}
 \|u-h\|_{L^\infty(B_R(x_0))}
 \leq CR^{(sp-d/q)/(p-1)}
 \|f-g\|_{L^q(B_R(x_0))}^{1/(p-1)},
\end{equation}
where $C$ depends only on $d,s,p,q$.
\end{lemma}

\begin{proof}

Existence follows from applying the direct method to the following strictly convex, continuous and coercive energy
\[
 \A(w):= \frac1{2p}\iint_{\R^{2d}\setminus({B_R(x_0)}^c\times {B_R(x_0)}^c)}
 \frac{|\delta u+\delta w|^p-|\delta u|^p}{|x-y|^{d+sp}}
 \,\dd x\,\dd y-\int_{B_R(x_0)}gw\,\dd x,
\]
defined for \(w\in W_0^{s,p}(B_R(x_0))\), where  $\delta v:=v(x)-v(y)$.

For more details, we refer to \cite[Proposition 2.12]{brasco2018higher}. Uniqueness follows easily, by taking the difference of two solutions as a test function.

We prove the estimate \eqref{eq:comparison-Lq} on $B_1$. Set $z=u-h$,
$F=\|f-g\|_{L^q(B_1)}$, and, for $k\geq0$, let
$A_k=\{x\in B_1:z(x)>k\}$ and $\psi=(z-k)_+$.
Since $z\in W^{s,p}_0(B_1)$ and $k\geq0$, this is an admissible test. Note that
\[
(\delta z)(\delta \psi)\geq 0,\quad |\psi(x)-\psi(y)|\leq|z(x)-z(y)|.
\]
Thus, applying \eqref{eq:monotonicity}, we get
\begin{align*}
    \iint\frac{(J_p(\delta u)-J_p(\delta h))(\delta\psi)}{|x-y|^{d+sp}}\dd x\dd y
    &=\iint\frac{(J_p(\delta u)-J_p(\delta h))(\delta z)}{|x-y|^{d+sp}}\frac{|\delta \psi|}{|\delta z|}\dd x\dd y\\
    &\geq C\iint \frac{|\delta z|^{p-1}|\delta \psi|}{|x-y|^{d+sp}}\dd x\dd y\geq C[\psi]^p_{W^{s,p}(\R^d)}.
\end{align*}
All the double integrals above are in $\R^{2d}\setminus(B_1^c\times B_1^c)$, and we used the convention that $(\delta z)|\delta \psi| /|\delta z|=0$ whenever $\delta z=0$. Combining this with the equations solved by $u$ and $h$, we obtain
\begin{equation}\label{eq:truncation-Lq}
 c[\psi]_{W^{s,p}(\R^d)}^p
 \leq\int_{B_1}|f-g|\psi\,\dd x.
\end{equation}
If $F=0$, testing also the negative part gives $z=0$. Assume $F>0$.

Suppose first that $sp\leq d$. If $sp<d$, take
$m=p_s^*$ defined in \eqref{eq:fractional-critical-exponent} and if $sp=d$, take $m=2pq'$. This choice implies, by Sobolev embedding of Lemma \ref{l:sob_embed},
\[
\|\psi\|_{L^m(B_1)}^p\leq C[\psi]^p_{W^{s,p}(\R^d)}\leq C\int_{A_k}|f-g|\psi\dd x.
\]
Applying H\"older inequality twice gives
\begin{align*}
    &\int_{A_k}|f-g|\psi \dd x\leq F\left(\int_{A_k} \psi^{q'} \dd x\right)^\frac{1}{q'}\leq F\left(\int \psi^m\dd x\right)^\frac{1}{m}|A_k|^{\frac{1}{q'}-\frac{1}{m}}.
\end{align*}
We used $q'<m$. Combining these estimates, we get 
\[
 \|\psi\|_{L^m(B_1)}^{p-1}
 \leq CF|A_k|^{1-1/q-1/m}.
\]
Moreover, for $t>k$, $A_t\subset A_k$ and $\psi>t-k$ in $A_t$. So
\[
\int_{A_k}\psi^m\dd x\geq |A_t|(t-k)^m.
\]
Therefore,
\[
 |A_t|\leq
 \left(\frac{CF^{1/(p-1)}}{t-k}\right)^m|A_k|^{1+\theta},
 \qquad
 \theta=\frac{m(1-1/q)-1}{p-1}-1>0.
\]
Taking
$k_j=M(1-2^{-j})$, we obtain the recursion
\[
|A_{k_{j+1}}| \leq \left( \frac{2C F^{1/(p-1)}}{M }\right)^m 2^{jm} |A_{k_{j}}|^{1+\theta}.
\]
We want to apply Lemma \ref{l:fast_convergence} to $Y_j:=|A_{k_{j+1}}|$. In order to satisfy the smallness condition on $Y_0$, we take $M=C_1F^{1/(p-1)}$ and $C_1$ sufficiently
large, so that
\[
|A_{k_1}|\leq \left(\frac{4CF^\frac{1}{p-1}}{M}\right)^{m}|B_1|^{1+\theta}=\left(\frac{4C}{C_1}\right)^m|B_1|^{1+\theta}
\]
is small enough. An application of this lemma gives $Y_j\to 0$ and thus $\|z_+\|_\infty\leq C_1F^{1/(p-1)}$, as intended.

If instead $sp>d$, \eqref{eq:fractional-Morrey} and
\eqref{eq:truncation-Lq} with $k=0$ give instead
\[
 \|z_+\|_{L^{\infty}(B_1)}^p
 \leq C[z_+]_{W^{s,p}(\R^d)}^p
 \leq CF\|z_+\|_{L^{q'}(B_1)}
 \leq CF\|z_+\|_{L^{\infty}(B_1)}.
\]
This proves the same estimate. Applying these arguments to $-z$
also bounds its negative part.

Finally, the change of variables $x=x_0+Ry$ replaces the source
difference by $R^{sp}(f-g)(x_0+Ry)$, whose $L^q(B_1)$ norm is
$R^{sp-d/q}\|f-g\|_{L^q(B_R(x_0))}$. This proves
\eqref{eq:comparison-Lq}.
\end{proof}

\section{\texorpdfstring{$C^{1,\alpha}$}{C1,alpha} with source term in \texorpdfstring{$L^q$}{Lq}}
\label{s:measurable}

We prove Theorem~\ref{t:measurable} using the comparison estimate as well as the $C^{1,\alpha_0}$ estimate for the homogeneous equation from \cite{GJS}.

\begin{proof}[Proof of Theorem~\ref{t:measurable}]
Let
\[
A:=\|u\|_{L^\infty(B_2)}
 +{\Tail}_{p-1,sp}(u;2)
 +\|f\|_{L^q(B_2)}^{1/(p-1)},
\]
and assume $0<A<\infty$.
The case $A=0$ is immediate. We first obtain a Lipschitz bound for $u$.

Notice that the range assumption implies \[ sp-\frac dq>p-1>0, \] and hence \(q>d/(sp)\), so that Lemma~\ref{l:comparison-Lq} is applicable.

From \eqref{eq:measurable-exponent}, recall that for $p\geq 2$
\[
\eta_q:=\frac{sp-d/q}{p-1}-1\in(0,1).
\]
Then \cite[Theorem 1.4]{brasco2018higher} gives, for every
$0<\theta<1$,
\begin{equation}\label{eq:almost-Lipschitz-Lq}
 [u]_{C^{0,\theta}(B_{7/4})}\leq C_\theta A.
\end{equation}
Indeed, its limiting H\"older exponent is $\min\{1+\eta_q,1\}=1$. Let $\alpha_0$ be the exponent corresponding to the  $C^{1,\alpha_0}$ regularity from \cite{GJS} and choose
\begin{equation}\label{eq:pick_theta}
 \theta=1-\frac{\alpha_0\eta_q}2\in(0,1).  
\end{equation}
 All constants below may
depend on this choice, hence only on $d,s,p,q$.

Fix $x_0\in B_{3/2}$ and $0<R\leq1/8$. Let $h_R$ be the
homogeneous replacement of $u$ in $B_R(x_0)$, given by
Lemma~\ref{l:comparison-Lq}. Then
\begin{equation}\label{eq:replacement-error-Lq}
 \|u-h_R\|_{L^\infty(B_R(x_0))}
 \leq C R^{1+\eta_q}\|f\|_{L^q(B_R(x_0))}^{1/(p-1)}
 \leq CA R^{1+\eta_q}.
\end{equation}
Combining this with \eqref{eq:almost-Lipschitz-Lq},  we obtain
\begin{equation}\label{eq:replacement-supremum-Lq}
 \begin{aligned}
     \|h_R-u(x_0)\|_{L^\infty(B_R(x_0))}&\leq \|h_R-u\|_{L^\infty(B_R(x_0))}+\|u-u(x_0)\|_{L^\infty(B_R(x_0))}\\
     &\leq CAR^{1+\eta_q}+C_\theta AR^\theta
 \leq CA R^\theta.
 \end{aligned}
\end{equation}
We now bound the tail of $h_R-u(x_0)$ outside $B_R(x_0)$, where $h_R=u$. For distances at most $1/8$, we use \eqref{eq:almost-Lipschitz-Lq}, with $sp>\theta(p-1)$. On the remaining region we use $\|u\|_{L^\infty(B_2)}$ and the tail of $u$, observing that $|y-x_0|\geq|y|/4$ for $y\notin B_2$. Hence we get
\begin{equation}\label{eq:replacement-tail-Lq}
    \begin{aligned}
        {\Tail}_{p-1,sp}(h_R-u(x_0);x_0,R)^{p-1}&=\int_{B^c_{1/8}(x_0)}\frac{|u(z)-u(x_0)|^{p-1}}{|z-x_0|^{d+sp}}\dd z+\int_{B_{1/8}(x_0)\setminus B_R(x_0)}\frac{|u(z)-u(x_0)|^{p-1}}{|z-x_0|^{d+sp}}\dd z\\
        &\leq C\left(\|u\|_{L^\infty(B_2)}^{p-1}+{\Tail}_{p-1,sp}(u;2)^{p-1}\right)+C_\theta A^{p-1}\int_R^{1/8}t^{\theta (p-1)-sp-1}\dd t\\
        &\leq CA^{p-1}R^{\theta(p-1)-sp}.
    \end{aligned}
\end{equation}

Apply \cite[Theorem 1.1]{GJS} to
\[
 H_R(z)=\frac{h_R(x_0+(R/2)z)-u(x_0)}{R^\theta},
\]
which solves
\[
(-\Delta_p)^sH_R(z)=0 \quad \text{ in } B_2,
\]
Moreover, from \eqref{eq:replacement-supremum-Lq} we get $\|H_R\|_{L^\infty(B_2)}\leq CA$ and rescaling \eqref{eq:replacement-tail-Lq} using Lemma \ref{l:tail_scale}, we obtain
\[
{\Tail}_{p-1,sp}(H_R;2)=2^{-\frac{sp}{p-1}}R^{\frac{sp}{p-1}-\theta}{\Tail}_{p-1,sp}(h_R-u(x_0);x_0,R)\leq CA.
\]
Scaling back the $C^{1,\alpha_0}$ estimate for $H_R$ gives
\begin{equation}\label{eq:replacement-gradient-Lq}
 \|\nabla h_R\|_{L^\infty(B_{R/2}(x_0))}
 +R^{\alpha_0}[\nabla h_R]_{C^{0,\alpha_0}(B_{R/2}(x_0))}
 \leq CA R^{\theta-1}.
\end{equation}
Therefore, for $0<r\leq R/2$, we can use \eqref{eq:replacement-error-Lq} and \eqref{eq:replacement-gradient-Lq} to get
\begin{equation}\label{eq:two-radii-Lq}
 \begin{aligned}
     &\sup_{B_r(x_0)}|u(x)-u(x_0)-\nabla h_R(x_0)\cdot(x-x_0)|\\
     &\leq 2\sup_{B_r(x_0)}|u-h_R|+\sup_{B_r(x_0)}|h_R(x)-h_R(x_0)-\nabla h_R(x_0)\cdot(x-x_0)|\\
 &\quad\leq CA\left(R^{1+\eta_q}+r^{1+\alpha_0}R^{\theta-1-\alpha_0}\right).
 \end{aligned}
\end{equation}
We balance the terms in \eqref{eq:two-radii-Lq} by choosing 
\[
 R=r^{(1+\alpha_0)/(2 + \alpha_0 + \eta_q - \theta)},
 \qquad
 \beta= \frac{\alpha_0\eta_q+\theta-1}{2+\eta_q+\alpha_0-\theta}>0,
\]
since $\theta$ is given by \eqref{eq:pick_theta}.
For a sufficiently small $r_*>0$, depending only on the parameters,
$r\leq r_*$ ensures $R\leq1/8$ and $2r\leq R$. Both terms on
the right of \eqref{eq:two-radii-Lq} then equal $r^{1+\beta}$. 

Note however that the obtained slope $\nabla h_R(x_0)$ depends on the scale $r$. Hence we proved only that for every $x_0\in B_{3/2}$ and $r\leq r_*$, there is a
slope $a_{x_0,r}$ such that
\begin{equation}\label{eq:affine-approximation-Lq}
 \sup_{B_r(x_0)}|u(x)-u(x_0)-a_{x_0,r}\cdot(x-x_0)|
 \leq CA r^{1+\beta}.
\end{equation}
For $x_0 \in B_{5/4}$, calling $\ell_{x,r}(y)=u(x)+a_{x,r}\cdot(y-x)$, we can apply Lemma \ref{l:affine-approximation} to conclude
\begin{equation}\label{eq:derived-Lipschitz-Lq}
 \|u\|_{C^{1,\beta}(B_{r_\ast}(x_0))}\leq CA,
 \qquad [u]_{\mathrm{Lip}(B_{r_\ast}(x_0))}\leq CA.
\end{equation}
Since $r_\ast$ depends only on the parameters, a simple covering argument concludes the result in $B_{5/4}$.
 
We can now repeat \eqref{eq:replacement-supremum-Lq} and
\eqref{eq:two-radii-Lq} with $\theta=1$, centers $x_0\in B_1$,
and $R\leq1/8$, using \eqref{eq:derived-Lipschitz-Lq}.
The tail integral is finite with the required bound because
$sp>p-1$. We obtain slopes bounded by $CA$ such that
\[
 \sup_{B_r(x_0)}|u(x)-u(x_0)-\nabla h_R(x_0)\cdot(x-x_0)|
 \leq CA\left(R^{1+\eta_q}+r^{1+\alpha_0}R^{-\alpha_0}\right).
\]
Taking $R=r^{(1+\alpha_0)/(1+\alpha_0+\eta_q)}$ makes both terms equal
$r^{1+\alpha}$, with $\alpha$ as in
\eqref{eq:measurable-exponent}. The same affine approximation
argument proves \eqref{eq:measurable-estimate}.
\end{proof}

\section{\texorpdfstring{$C^{1,\alpha}$}{C1,alpha} with H\"older source term}
\label{s:holder}

We now prove Theorem~\ref{t:holder}. We first refine the estimate from \cite[Theorem 1.1]{GJS} to include constant source term and depend instead on the smaller tail of order $sp+1$. We then freeze the right-hand side and compare
on two different scales.

The following estimate is a consequence of the proof in \cite{GJS}.

\begin{lemma}\label{l:homogeneous-tail}
Let $0<s<1$, $2\leq p<2/(1-s)$ and $c\in \R$. There are
$\alpha_0\in(0,1)$ and $C>0$, depending only on $d,s,p$, such that
every weak solution
$v\in W^{s,p}_{\mathrm{loc}}(B_2)\cap L^{p-1}_{sp}(\R^d)$ of
$\dps v=c$ in $B_2$ satisfies
\begin{equation}\label{eq:homogeneous-tail}
 \|v\|_{C^{1,\alpha_0}(B_1)}
 \leq C\left(\|v\|_{L^\infty(B_2)}
 +{\Tail}_{p-1,sp+1}(v;2)\right).
\end{equation}
Moreover, $\alpha_0$ and $C$ are independent of $c$.
\end{lemma}

\begin{proof}
We start by remarking that the whole argument of \cite{GJS} can be repeated for equations with constant right-hand side, and the estimate obtained is independent of this constant. More precisely, in the first stage of the proof, the equation is differentiated and therefore the constant disappears immediately. The second stage consists of an Ishii-Lions argument, where two points $\bar x, \bar y$ are constructed and the only instance where we use the equation is to evaluate 
\[
\dps u(\bar x)-\dps u(\bar y)=c-c=0,
\]
hence here also the constant disappears. In the third stage, as in the first, we differentiate the equation. Therefore, the result of \cite[Theorem 1.1]{GJS} holds also for constant right-hand sides.

The desired bound follows immediately by interpolation from the following bound, established in the proof of Theorem 1.1 in \cite[Sections 8--9]{GJS}
\[
\|v\|_{C^{1,\alpha_0}(B_1)}\leq C (\| v\|_{\mathrm{Lip}(B_{1})}+{\Tail}_{p-1,sp+1}(v;1) ).
\]
Indeed, by interpolation, we have for every $\varepsilon>0$,
\[
\|v\|_{\mathrm{Lip}(B_1)}\leq \varepsilon\|v\|_{C^{1,\alpha_0}(B_1)}+C_\varepsilon\|v\|_{L^\infty(B_1)}.
\]
Choosing $\varepsilon>0$ sufficiently small, the first term is absorbed into the left-hand side.
Moreover, 
\[
{\Tail}_{p-1,sp+1}(v;1)\leq {\Tail}_{p-1,sp+1}(v;2)+C\|v\|_{L^\infty(B_2)}.
\]
Hence \eqref{eq:homogeneous-tail} follows from combining these three estimates.
\end{proof}

We prove the second main theorem.

\begin{proof}[Proof of Theorem~\ref{t:holder}]

Denote
\[
A:=\|u\|_{L^\infty(B_2)}
 +{\Tail}_{p-1,sp}(u;2)
 +\|f\|_{C^{0,\gamma}(B_2)}^{1/(p-1)}.
\]
We assume $0<A<\infty$; the case $A=0$
is immediate.

The assumptions $sp+\gamma>p-1$ and $\gamma\leq1$ imply
$sp>p-1-\gamma\geq p-2$. Thus \cite[Theorems 1.1--1.2]{biswas2025improved} apply and give
\begin{equation}\label{eq:lipschitz}
 [u]_{\mathrm{Lip}(B_{3/2})}\leq CA.
\end{equation}
For the quantitative bound, apply that result to $u/A$ and
$f/A^{p-1}$, whose data are bounded by one, and use a fixed interior
covering. The constants depend only on $d,s,p,\gamma$.

Fix $x_0\in B_{5/4}$ and $0<R\leq1/4$. Let $h_R$ be the solution of
\[
 \dps h_R=f(x_0)\quad\text{in }B_R(x_0),\qquad
 h_R=u\quad\text{in }B_R(x_0)^c,
\]
provided by Lemma~\ref{l:comparison-Lq} with $q=\infty$. By the
H\"older continuity of $f$ and \eqref{eq:comparison-Lq}, we get
\begin{equation}\label{eq:freezing-error}
    \begin{aligned}
 \|u-h_R\|_{L^\infty(B_R(x_0))}
 &\leq CR^{sp/(p-1)}
 \|f-f(x_0)\|_{L^\infty(B_R(x_0))}^{1/(p-1)}\\
 &\leq C[f]_{C^{0,\gamma}(B_2)}^{1/(p-1)}
 R^{(sp+\gamma)/(p-1)}
 \leq CA R^{1+\eta_\gamma}.
\end{aligned}
\end{equation}
This, together with \eqref{eq:lipschitz}, gives
\begin{equation}\label{eq:replacement-supremum}
 \|h_R-u(x_0)\|_{L^\infty(B_R(x_0))}\leq CA R.
\end{equation}

We next estimate its exterior values. Since $h_R=u$ outside the
ball, arguing as in \eqref{eq:replacement-tail-Lq} using \eqref{eq:lipschitz} yields
\begin{equation}\label{eq:replacement-tail}
    \begin{aligned}
 {\Tail}_{p-1,sp+1}&(h_R-u(x_0);x_0,R)^{p-1}
 =\int_{B_R(x_0)^c}
 \frac{|u(y)-u(x_0)|^{p-1}}{|y-x_0|^{d+sp+1}}\,\dd y\\
 &\leq CA^{p-1}
 \left(\int_R^{1/4}t^{p-sp-3}\,\dd t+1\right)
 \leq CA^{p-1}R^{p-sp-2}.
\end{aligned}
\end{equation}
Consider
\[
 H_R(z)=\frac{h_R(x_0+(R/2)z)-u(x_0)}{R}.
\]
Then
\[
(-\Delta_p)^sH_R(z)=2^{-sp}R^{sp-p+1}f(x_0)\quad \text{ for } z\in B_2.
\]
By
\eqref{eq:replacement-supremum}--\eqref{eq:replacement-tail} and Lemma \ref{l:tail_scale}, we have
\[
 \|H_R\|_{L^\infty(B_2)}
 +{\Tail}_{p-1,sp+1}(H_R;2)\leq CA.
\]
Hence, Lemma~\ref{l:homogeneous-tail}
gives
\begin{equation}\label{eq:replacement-gradient}
 \|\nabla h_R\|_{L^\infty(B_{R/2}(x_0))}
 +R^{\alpha_0}[\nabla h_R]_{C^{0,\alpha_0}(B_{R/2}(x_0))}
 \leq CA.
\end{equation}

For $0<r\leq R/2$, combining
\eqref{eq:freezing-error} with Taylor's formula and
\eqref{eq:replacement-gradient}, we obtain
\begin{equation}\label{eq:two-radii}
 \sup_{B_r(x_0)}|u(x)-u(x_0)-\nabla h_R(x_0)\cdot(x-x_0)|
 \leq CA\left(R^{1+\eta_\gamma}+R^{-\alpha_0}r^{1+\alpha_0}\right).
\end{equation}
Choose
\[
 R=r^{(1+\alpha_0)/(1+\alpha_0+\eta_\gamma)}.
\]
There is a universal $r_*>0$ such that $r\leq r_*$ implies
$R\leq1/4$ and $2r\leq R$. For this choice, both terms in
\eqref{eq:two-radii} equal $r^{1+\alpha}$, with $\alpha$ as in
\eqref{eq:exponent}. Thus, for every $x_0\in B_{5/4}$ and $r\leq r_*$,
there is a vector $a_{x_0,r}$ such that
\begin{equation}\label{eq:affine-approximation}
 |a_{x_0,r}|\leq CA,\qquad
 \sup_{B_r(x_0)}|u(x)-u(x_0)-a_{x_0,r}\cdot(x-x_0)|
 \leq CA r^{1+\alpha}.
\end{equation}

The affine approximation result of Lemma \ref{l:affine-approximation} applied for $x_0\in B_{1}$ to
\eqref{eq:affine-approximation}, gives
\[
 \|\nabla u\|_{L^\infty(B_{r_\ast}(x_0))}+[\nabla u]_{C^{0,\alpha}(B_{r_\ast}(x_0))}\leq CA.
\]
Since $r_\ast$ depends only on the parameters, a simple covering argument gives the same result in $B_1$.
Together with the bound for $u$, this proves
\eqref{eq:main-estimate}.
\end{proof}

\bigskip

\noindent\textbf{Acknowledgements. } The authors would like to thank Luis Silvestre for fruitful conversations on the topic of this paper. DG is a  member of the Gruppo Nazionale per l'Analisi Matematica, la
Probabilità e le loro applicazioni (GNAMPA) - INdAM. D.J. is supported by the European Research Council, through the ERC StG project NEW, No.~101220121.

\bigskip

\noindent\textbf{AI declaration. } The authors started working on this problem soon after the paper \cite{GJS} was released. They developed a draft of the manuscript with bounded right-hand side, before using ChatGPT Astra 6.0. Subsequent discussions with the AI led to the current formulation of Lemma~\ref{l:comparison-Lq}, which allowed the authors to simplify the original argument. ChatGPT Astra 6.0 was also used to assist with the revision and final preparation of the manuscript. All mathematical statements, proofs, and computations were independently verified by the authors, who take full responsibility for the content of the paper.

\bibliographystyle{plain}
\bibliography{bib_GJ}
{\raggedright\Addresses\par}
\end{document}